\documentclass[11pt]{amsart}
\usepackage{amssymb,amsmath,amsfonts,amsthm}
\usepackage{graphicx}
\usepackage{color}
\usepackage{psfrag}

\newtheorem{theorem}{Theorem}[section]
\newtheorem{lemma}[theorem]{Lemma}
\newtheorem{proposition}[theorem]{Proposition}
\newtheorem{corollary}[theorem]{Corollary}
\theoremstyle{definition}

\def\real{\mathbb{R}}

\def\natural{\mathbb{N}}
\def\supp{\operatorname{supp}}
\def\d{\operatorname{dist}}

\def\id{\operatorname{id}}

\def\cW{\mathcal{W}}

\def\rels{\sim^s}
\def\relu{\sim^u}

\def\cH{\mathcal{H}}

\def\cE{\mathcal{E}}

\def\cB{\mathcal{B}}

\def\quand{\quad\text{and}\quad}

\def\tm{\tilde{m}}

\newcommand{\norm}[1]{{\left\lVert  #1  \right\rVert}}
\newcommand{\abs}[1]{{\left\lvert  #1  \right\rvert}}

\title[]
{notes on $s$ and $u$-states for cocycles over partially hyperbolic maps}
\author{Mauricio Poletti}
\address{LAGA -- Universit\'e Paris 13, 99 Av. Jean-Baptiste Cl\'ement, 93430 Villetaneus, France.}
\email{mpoletti@impa.br}

\begin{document}

\begin{abstract}
In this notes we prove that the $s$ or $u$-states of cocycles over partially hyperbolic maps are closed in the space of invariant measures.
\end{abstract}

\maketitle

\section{introduction}
In the study of Lyapunov exponents of linear cocycles over hyperbolic or partially hyperbolic maps one of the principal tools to prove positivity, simplicity or continuity is to analyse the invariant measures of the cocycle that projects to some fixed invariant measure in the base.

With some conditions that allow the existence of linear stable and unstable holonomies, having zero exponents (in some cases also discontinuity) can be caracterized by some rigidity condition in the invariant measures of the cocycles, this is known as the \emph{Invariance Principle} (see \cite{Extremal}). This rigidity condition says that the measures must be $s$ and $u$-states, this means that the disintegration is invariant by the holonomies (see section~\ref{s.cocycles} for the precise definition).

In many works closeness of $s$ or $u$-states has been proved and used for specific cases (\cite{BBB}, \cite{Extremal}, \cite{Pol16}, \cite{ASV13}). The purpose of this notes is to give a more general proof of this fact for general partially hyperbolic maps without any extra conditions on invariant measure of the base map. 

The precise statement of the main result is given in theorem~\ref{teo}.

\textbf{Acknowledgements.} Thanks to Mateus Sousa for the idea of the proof of lemma~\ref{l.closed}.

\section{Partial hyperbolicity}\label{s.PH}
A diffeomorphism $f:M \to M$ of a compact $C^k$, $k>1$, manifold $M$ is said to be \emph{partially hyperbolic} if there
exists a non-trivial splitting of the tangent bundle
\begin{equation*}
TM=E^s\oplus E^c\oplus E^u
\end{equation*}
invariant under the derivative $Df$, a Riemannian metric $\|\cdot\|$ on $M$, and positive
continuous functions $\nu$, $\hat{\nu}$, $\gamma$, $\hat{\gamma}$ with $\nu$, $\hat{\nu}<1$
and $\nu<\gamma<{\hat\gamma}^{-1}<{\hat\nu}^{-1}$ such that, for any unit vector $v\in T_pM$,
\begin{alignat*}{2}
& \|Df(p)v \| < \nu(p) & \quad & \text{if } v\in E^s(p),
 \\
\gamma(p) < & \|Df(p)v \|  < {\hat{\gamma}(p)}^{-1} & & \text{if } v\in E^c(p),
 \\
{\hat{\nu}(p)}^{-1} < & \|Df(p)v\| & &  \text{if } v\in E^u(p).
\end{alignat*}
All three sub-bundles $E^s$, $E^c$, $E^u$ are assumed to have positive dimension. From now on,
we take $M$ to be endowed with the distance $\d:M\times M\to \real$ associated to such a Riemannian structure.

Suppose that $f:M\to M$ is a partially hyperbolic diffeomorphism. The stable and unstable bundles $E^s$ and $E^u$
are uniquely integrable and their integral manifolds form two transverse continuous foliations
$\cW^s$ and $\cW^u$, whose leaves are immersed sub-manifolds of the same class of
differentiability as $f$. These foliations are referred to as the \emph{strong-stable} and
\emph{strong-unstable} foliations. They are invariant under $f$, in the sense that
$$
f(\cW^s(x))= \cW^s(f(x)) \quand f(\cW^u(x))= \cW^u(f(x)),
$$
where $\cW^s(x)$ and $\cW^s(x)$ denote the leaves of $\cW^s$ and $\cW^u$, respectively,
passing through any $x\in M$. We denote by $x\relu y$ if $x\in \cW^u(y)$, and analogously $x\rels z$ if there are in the same stable manifold.

\section{cocycles}\label{s.cocycles}
Let $\cE$ be a compact manifold, and let $F:M\times \cE\to M\times \cE$ be a cocycle over $f$, this means that if $P:M\times \cE\to M$ is the natural projection to the first coordinate $P\circ F=f\circ P$ and $x\mapsto F_x$ is H\"older continuous to the topology of $C^r$ diffeomorphisms.

We say that $F$ admits \emph{stable holonomies} if for every $x,y\in M$, $x\rels y$, there exists $H^s_{x,y}:\cE\to\cE$ with the following properties: 
\begin{itemize}
\item $x\rels y\rels z$, $H^s_{x,z}=H^s_{y,z}\circ H^s_{x,y}$ and $H^s_{x,x}=\id$,
\item $F_y\circ H^s_{x,y}=H^s_{f(x),f(y)}\circ F_x$,
\item $(x,y,\xi)\to H^s_{x,y}(\xi)$ is continuous where $(x,y)$ varies in the set of points $x\rels y$,
\item there exist $C>0$ and $\gamma>0$ such that $ H^s_{x,y}$ is $(C,\gamma)$ H\"older for every $x\rels y$. 
\end{itemize}
Analogously we say that $F$ admits \emph{unstable holonomies} if for every $x\relu y$ there exist $ H^u_{x,y}$ with the same properties changing stable by unstable.
 
From now on fix $f$ and vary the cocycles $F$ projecting to $f$ in a topology such that $x,y,F\mapsto H^{s,F}_{x,y}$ varies continuously. 

Fix some $f$-invariant probability measure $\mu$, as $\cE$ is compact there always exists some $F$-invariant probability measure $m$ that projects to $\mu$. By Rokhlin disintegration theorem, we can disintegrate $m$ with respect to the partition given by the fibers $\{x\} \times \cE$, so we have $x\mapsto m_x$ defined almost everywhere.  

We say that an $F$-invariant measure that projects to $\mu$ is an \emph{$s$-state} if there exists a total measure subset $M'\subset M$ such that for every $x,y \in M'$, $x\rels y$, ${H^s_{x,y}}_*m_x=m_y$. Analogously, we say that a measure is an \emph{$u$-state} is the same is true changing stable by unstable manifolds. We call $m$ an $su$-state if it is booth $s$ and $u$-state.

We want to prove that
\begin{theorem}\label{teo}
If $m^k$ are $s$-states for $F_k$, that projects to $\mu$ such that $F_k\to F$ and $m^k\to m$  in the weak$^*$ topology then $m$ is an $s$-state. 
\end{theorem}

By \cite[theorem~4.1]{ASV13} if a cocycle $F$ has all his Lyapunov exponents equal to zero, then the $F$-invariant measure $m$ is an $su$-state. As a corollary we have
\begin{corollary}
If $F$ does not admit any $su$-state, then there exists a neighbourhood of $F$ with non-zero exponents.
\end{corollary}
\section{proof}

First we need to recall the Markov consturction of \cite{ASV13}:
Given any point $x\in \supp(\mu)$ we can find some section $\Sigma$ transverse to the stable foliation, some $N>0$, $R>0$, $0<\delta<\frac{R}{2}$ and a measurable family $\lbrace S(z),z\in \Sigma \rbrace$ such that 
\begin{itemize}
\item $\cW^s(z,\delta)\subset S(z)\subset \cW^s(z,R)$ for all $z\in \Sigma$,
\item for all $l\geq 1$, $z,\zeta\in \Sigma$, if $f^{lN}(S(z))\cap S(\zeta)\neq \emptyset$ then $f^{lN}(S(z))\subset S(\zeta)$.
\end{itemize}
As taking an iterate will not affect our argument we suppose that $N=1$.

For each $z\in \Sigma$, let $r(z)$ be the largest integer such that $f^j(S(z))$ does not intersect any $S(w)$ for all $w\in \Sigma$, $0<j\leq r(z)$. Now let $\cB_0$ be the $\sigma$-algebra of sets $E\subset M$ such that for every $z$ and $j$ as before, either $E$ contains $f^j(S(z))$ or is disjoint from it. A $\cB_0$-mensurable function, is a function that is constant on the sets $f^j(S(z))$, $0\leq j\leq r(z)$.

For every $k\in\natural$, let $\cH_k:M\times\cE\to M\times \cE$ be defined by $(x,\xi)\to (x,{H_k}_x(\xi))$ where
\begin{equation}
{H_k}_x=\left\lbrace\begin{array}{cc}
H^{s,k}_{x,f^j(z)} & \text{ if }x\in f^j(S(z))\text{ for some }z\in \Sigma \\
\id & \text{otherwise}
\end{array}\right.
\end{equation}
where $H^{s,k}_{x,z}$ is the stable holonomie of $F_k$.

Now as in \cite{ASV13} we can change our cocycle by
$\tilde{F}_k=\cH_k F_k (\cH_k)^{-1}$, this is called a deformation cocycle of $F_k$, such that $x\mapsto \tilde{F_k}_x$ is $\cB_0$ measurable.

Let $m^k$ be an $F_k$-invariant measure, define $\tm^k=\cH^k_* m^k$, this measure is $\tilde{F}_k$-invariant.
Observe that $m$ being an $s$-state implies  that $x\mapsto m^k_x$ is $\cB_0$ measurable. Moreover, $m^k$ is an $s$-state if and only if this is true for every $z\in M$ and $\Sigma\ni z$ transversal to the stable foliation (this is explained in more detail in \cite[Section~4.4]{ASV13}).

\begin{lemma}\label{l.measurable1}
Let $\phi:M\times\cE\to \real$ be a measurable bounded function such that $x\mapsto \phi(x,v)$ continuous, then $\int \phi d m^k\to \int \phi d m$.
\end{lemma}
\begin{proof}
Fix $\varepsilon>0$ and take a compact set $K\subset M$ such that $\mu(K)>1-\frac{\varepsilon}{\norm{\phi}}$ and $\phi$ is continuous in $K\times \cE$, take   $\phi':M\times \cE \to \real$ be a continuous function such that $\phi(x,v)=\phi'(x,v)$ for every $x\in K$, $v\in \cE$ and $\norm{\phi'}\leq \norm{\phi}$. 

Now
take $k$ sufficiently large such that $\abs{\int \phi' d m^k- \int \phi' d m}<\epsilon$, then 
$$
\abs{\int \phi d m^k- \int \phi d m}\leq \abs{\int \phi' d m^k- \int \phi' d m}+2 \varepsilon.
$$

So for $k$ sufficiently large this is less than $3\varepsilon$, concluding the proof.
\end{proof}

\begin{lemma}
If $m^k\to m$ then $\tm^k\to \tm=\cH_* m$.
\end{lemma}
\begin{proof}
Let $\varphi:M\times \cE\to \real$ be a continuous function, then
$$
\int \varphi d\tm^k=\int \varphi\circ \cH_k d m^k,
$$
$\varphi\circ \cH$ is measurable but $v\mapsto \varphi\circ \cH(x,v)$ is continuous for every $x\in M$, then by lemma~\ref{l.measurable1} we have that 
\begin{equation}\label{eq.limite}
\int \varphi\circ \cH d m^k\to \int \varphi d \tm.
\end{equation}

Fix some $\epsilon>0$ and take $\delta>0$ such that $\d(a,b)<\delta$ implies that $\abs{\varphi(a)-\varphi(b)}<\epsilon$.
Now, the uniform convergence of the holonomies implies that for $k$ sufficiently large 
\begin{equation}\label{eq.uniform}
\d(\cH_k(x,v),\cH(x,v))<\delta.
\end{equation}

Observe that
$$
\abs{\int \varphi\circ \cH_k d m^k- \int \varphi d \tm}\leq 
\abs{\int \varphi\circ \cH_k d m^k- \int \varphi\circ\cH d m^k}+\abs{\int \varphi\circ \cH d m^k- \int \varphi d \tm}
$$
then by \eqref{eq.limite} and \eqref{eq.uniform} we have that for $k$ sufficiently large 
$$
\abs{\int \varphi\circ \cH_k d m^k- \int \varphi d \tm}<2\varepsilon.
$$
\end{proof}

So we are left to prove that $\tm^k\to \tm$ implies that $x\mapsto \tm_x$ is also $\cB_0$ measurable.

\subsection{$\cB_0$ measurable}

Let $\cB_0$ be a $\sigma$-algebra, let $\mu$ be a measure in $M$. Assume that we have some measures $\tm^k$ in $M\times \cE$ converging in the weak$^*$ topology to $\tm$ and let $P:M\times \cE \to M$ be the natural projection, also assume that $P_* \tm^k=\mu$. 

The next lemma is a corollary of lemma~\ref{l.measurable1}.
\begin{lemma}\label{l.measurable}
Let $\phi:M\to \real$ be a measurable function and $\varphi:\cE\to \real$ continuous, then $\int \phi\times \varphi d \tm^k\to \int \phi \times \varphi d\tm$.
\end{lemma}
%
Suppose that $x\mapsto \tm^k_{x}$ is $\cB_0$ measurable, this is true if and only if for every continuous function $\varphi:\cE\to \real$, $x\mapsto \int \varphi d\tm^k_{x}$ is $\cB_0$ measurable.

First we need the next lemma
\begin{lemma}\label{l.closed}
Let $\phi_k$ be a sequence of function in $L^2(\mu)$ that is $\cB_0$ measurable such that $\phi_k$ converges weakly to $\phi$, then $\phi$ is $\cB_0$ measurable.
\end{lemma}
\begin{proof}
First observe that the space of $\cB_0$ measurable functions is closed and convex in $L^2(\mu)$, lets call this space by $H\subset L^2(\mu)$. Suppose that $\phi\notin H$ then by Hahn-Banach there exist some $\rho\in L^2(\mu)$ such that $\int \rho \xi d\mu=0$ for every $\xi\in H$ and $\int \rho \phi d\mu>0$.
A contradiction because $\int \rho \phi_k d\mu\to\int \rho \phi d\mu$
\end{proof}
Now to conclude the proof of theorem~\ref{teo} we prove:
\begin{proposition}
$x\mapsto \int \varphi d\tm_{x}$ is $\cB_0$ measurable.
\end{proposition}
\begin{proof}
Take $\phi_k(x)=\int \varphi d\tm^k_{x}$, this function bounded then it is in $L^2(\mu)$. For any function $\rho\in L^2(\mu)$ we have that 
$$
\int \rho\phi_k d \mu=\int \rho(x)\int \varphi(v)d\tm^k_{x}(v)d\mu=\int \rho\times \varphi d\tm^k
$$
So by lemma~\ref{l.measurable} we have that 
$$
\int \rho\phi_k d \mu\to \int \rho\phi d \mu,
$$
where $\phi(x)=\int \varphi d\tm_{x}$. By hypothesis $\phi_k$ is $\cB_0$ measurable, then by lemma~\ref{l.closed} $\phi$ is $\cB_0$ measurable.
\end{proof}

\bibliography{bib}
\bibliographystyle{plain}

\end{document}